\documentclass[11pt]{amsart}
\usepackage{verbatim, amscd, epsfig,amsmath,amsthm,amstext,amssymb,hyperref,bbm 
}
\usepackage[all]{xy}

\theoremstyle{plain}
\newtheorem{theorem}{Theorem} 

\newtheorem{lemma}[theorem]{Lemma}

\newtheorem{rem}[theorem]{Remark}

\theoremstyle{definition}
\newtheorem{definition}[theorem]{Definition}

\newtheorem{conj}{Conjecture}

\renewcommand{\Re}{{\rm Re}\,}
\renewcommand{\Im}{{\rm Im}\,}
\newcommand{\del}{\partial}
\newcommand{\delbar}{\bar{\del}}
\newcommand{\R}{\mathbb{ R}}
\newcommand{\C}{\mathbb{ C}}
\newcommand{\Z}{\mathbb{ Z}}

\newcommand{\Hh}{\mathcal{H}}

\renewcommand{\H}{\mathbb{ H}}

\renewcommand{\P}{\mathbb{ P}}
\newcommand{\HP}{\H\P}
\newcommand{\CP}{\C\P}
\newcommand{\invers}{^{-1}}

\DeclareMathOperator{\End}{End}
\DeclareMathOperator{\Hom}{Hom}

\DeclareMathOperator{\Gl}{GL}
\DeclareMathOperator{\tr}{tr}
\DeclareMathOperator{\im}{im}
\DeclareMathOperator{\Span}{span}

\newcommand{\zo}{^{(0,1)}}
\newcommand{\oz}{^{(1,0)}}

\newcommand{\trivial}[1]{\underline{\H}^{#1}}
\newcommand{\trivialC}[1]{\underline{\C}^{#1}}

\renewcommand{\line}{\text{---}}

\begin{document}
\title{Harmonic map methods for Willmore surfaces}
\author{K. Leschke}

\address{K. Leschke, Department of Mathematics, University of
  Leicester, University Road, Leicester LE1 7RH, United Kingdom}

\email{k.leschke@le.ac.uk}

\thanks{Author partially supported by DFG SPP 1154 ``Global Differential Geometry''}

\maketitle
 
\begin{abstract}
  In this note we demonstrate how the analogy between the harmonic
  Gauss map of a constant mean curvature surface and the harmonic
  conformal Gauss map of a Willmore surface can be used to obtain  
  results on Willmore surfaces.
\end{abstract}
\section{Introduction}

For an immersion $f: M \to \R^3$ of a Riemann surface into Euclidean
3--space the harmonicity of the Gauss map $N: M \to S^2$ characterises
by the Ruh--Vilms theorem \cite{Ruh-Vilms} a constant mean curvature
surface. In particular, methods from harmonic map theory can be used
to study constant mean curvature surfaces, e.g. \cite{eells&wood,
  hitchin-harmonic, Bob, PS}.  Viewing Euclidean 3--space as the
imaginary quaternions equipped with the inner product $<a,b> =
-\Re(ab)$, $a,b\in\Im\H$, the Gauss map of a constant mean curvature
surface can be seen as a complex structure on the trivial $\H$--bundle
over $M$.

For Willmore surfaces $f: M \to \R^4$ a similar characterisation is
known \cite{ejiri, rigoli}: a conformal immersion is Willmore if and
only if its conformal Gauss map is harmonic.  Since the Willmore
property is a conformal invariant, we study Willmore surfaces up to
M\"obius transformations, i.e., we consider conformal immersions $f: M
\to S^4$ from a Riemann surface into the 4--sphere. Our main tools are
from Quaternionic Holomorphic Geometry \cite{icm, coimbra, klassiker},
in particular, we view the quaternionic projective line $\HP^1=S^4$ as
the 4--sphere and use the general linear group $\Gl(2,\H)$ to study
the conformal geometry of the 4--sphere. In this setup, the conformal
Gauss map can be seen as a complex structure on the trivial
$\H^2$--bundle over $M$, and in this sense, Willmore surfaces can be
considered as a generalisation of a rank 1 theory to a rank 2 problem.
The aim of this paper is to demonstrate in two examples how this
analogy can be exploited to obtain results for Willmore surfaces.

\section{Willmore surfaces in the 4--sphere}

We first recall some basic facts on Willmore surfaces, for details
compare \cite{coimbra}.  In this paper we consider conformal
immersions $f: M \to S^4$ from a Riemann surface $M$ into $S^4$ where
we identify the 4--sphere $S^4=\HP^1$ with the quaternionic projective
line.  Since a point in $\HP^1$ gives a line in $\H^2$, a map $f: M
\to \HP^1$ can be identified with a line subbundle $L\subset \trivial
2$ of the trivial bundle $\trivial 2 = M \times \H^2$ by $L_p =
f(p)$. The differential of $f$ is under this identification given by
\[
\delta: \Gamma(L) \to
\Omega^1(\trivial 2/L), \psi \mapsto \pi d\psi
\]
where $d$ is the trivial connection on $\trivial 2$ and $\pi: \trivial
2\to \trivial 2/L$ is the canonical projection. An immersion is
\emph{conformal} if there exists a complex structure
$S\in\Gamma(\End(\trivial 2))$, that is a smooth map into the
quaternionic endomorphism with $S^2=-1$, which stabilises $L$ and is
compatible with the complex structure $J_M$ on the tangent space $TM$,
that is,
\begin{equation}
\label{eq:conformality}
*\delta = S \delta = \delta S
\end{equation}
where $*\delta(X) = \delta(J_M X)$ for $X\in TM$. Note that since $S L
\subset L$, the complex structure $S$ induces a complex structure on
$\trivial 2/L$ which we denote again by $S$.  The condition
(\ref{eq:conformality}) does not determine the complex structure $S$
uniquely; to fix $S$ we consider the conformal Gauss map of $f$: we
decompose the differential $dS= (dS)' + (dS)''$ of a complex structure
$S$ into $(1,0)$ and $(0,1)$ parts
\[
(dS)' = \frac 12(dS - S*dS), \quad (dS)''= \frac 12(dS + S*dS)
\]
with respect to the complex structure $S$. Denoting by
\begin{equation}
\label{eq:Hopf fields}
A = \frac 12(*dS)', \quad Q= -\frac 12(*dS)''\,,
\end{equation}
the \emph{Hopf fields} of $S$, the derivative of $S$ can be written as
$ dS = 2(*Q -*A)$. Note that the Hopf fields anti--commute with $S$
and thus satisfy
\begin{equation}
\label{eq:A,Q}
*A = S A = -AS\,, \quad \text{ and } \quad *Q = -S Q = Q S\,.
\end{equation}

\begin{definition}
  Let $f: M \to S^4$ be a conformal immersion from a Riemann surface
  $M$ into the 4--sphere. The \emph{conformal Gauss map} of $f$ is a
  complex structure $S\in \Gamma(\End(\trivial 2)), S^2=-1$, such that
$*\delta = S \delta = \delta S$ and 
\begin{equation}
\label{eq:conformal gauss map}
\im A \subset L \quad \text{or, equivalently, } \quad L \subset \ker Q\,.
\end{equation}
\end{definition}

A complex structure $S\in \End(\H^2)$ can be identified with a
2--sphere $S'\subset S^4$ by $S' =\{ l\in \HP^1 \mid Sl =l\}$. This
way, a complex structure $S\in\Gamma(\End(\trivial 2))$ gives a sphere
congruence, and the condition $SL \subset L$ says that $f(p) \in
S'(p)$ is a point on the sphere given by $S$ at $p$.  Moreover, the
conformality equation (\ref{eq:conformality}) says that $S$ envelopes
$f$ that is, $f(p)\in S'(p)$, and the tangent space of $S'(p)$ at
$f(p)$ and $d_pf(T_pM)$ coincide in an oriented way. Finally, the
condition (\ref{eq:conformal gauss map}) shows that the mean curvature
vectors of $f(M)$ and of $S'(p)$ coincide at $f(p)$. In other words,
the conformal Gauss map is the \emph{mean curvature sphere congruence}
of the conformal immersion $f$.

\begin{definition}
  A conformal immersion $f: M \to S^4$ of a compact Riemann surface
  $M$ into the 4--sphere is called \emph{Willmore surface} if $f$ is a
  critical point of the \emph{Willmore energy}
\[
W(f) =2 \int_M <A \wedge *A>
\]
under compactly supported variations of $f$ (where the conformal
structure of $M$ may change under the variation). Here $A$ is the Hopf
field of the conformal Gauss map $S$ of $f$, and $<B> =\tr B$ is the
real trace of an endomorphism $B\in\End(\H^2)$.
\end{definition}
\begin{rem}
  The Willmore energy of a conformal immersion $f: M \to S^4$ is given
  by
\[
W(f) =\int_M (|\Hh|^2- K -K^\perp)\text{vol}_{f^*h}
\]
where the mean curvature vector $\Hh$, the Gaussian curvature $K$, and
the normal bundle curvature $K^{\perp}$ are computed with respect to a
conformally flat metric $h$ on $S^4$.
\end{rem}
Since the energy functional of the conformal Gauss map coincides with
the Willmore energy of $f$ up to topological constants, one can show:
\begin{theorem}[\cite{ejiri, rigoli}, for the quaternionic formulation
  \cite{coimbra}]
  A conformal immersion $f: M \to S^4$ is Willmore if and only if the
  conformal Gauss map $S$ of $f$ is harmonic, that is, if and only if
\begin{equation}
\label{eq: conformal Gauss map harmonic}
d*A =0 \quad \text{ or, equivalently, } \quad d*Q =0\,.
\end{equation}

\end{theorem}

\section{Willmore sequences}

In the case of harmonic maps $N: M \to \CP^n$ of a Riemann surface
into complex projective space the $(0,1)$ and the $(1,0)$ part of the
derivative of $N$ give a sequence of harmonic maps
\cite{glaser-stora},\cite{din-zak}. This sequence can be used to prove
the Eells--Wood theorem \cite{eells&wood} and its generalisations
\cite{durhamboys, wolfson, uhlenbeck}): a harmonic map from a genus
$g$ Riemann surface into the complex projective line $\CP^1$ is
holomorphic or anti--holomorphic if the degree of the harmonic map is
bigger than $g-1$. For example, a constant mean curvature sphere has a
holomorphic unit normal map and thus is a round sphere as first
observed by Hopf \cite{hopf}.  On the other hand, immersed constant
mean curvature tori have unit normal maps of degree zero and the
harmonic sequence does not terminate. This case leads to the theory of
spectral curves and the construction of constant mean curvature tori
from algebraically completely integrable systems \cite{PS,
  hitchin-harmonic, Bob}.

We explain a corresponding $\delbar$ and $\partial$ construction on
the harmonic conformal Gauss map to obtain Willmore sequences,
\cite{sequences}. Using the Willmore sequence we get a unified proof
of the classification results for Willmore spheres \cite{bryant,
  montiel, ejiri}, and Willmore tori with non--trivial normal bundle
\cite{willmore_tori}. In particular, it only remains to study
integrable system methods and spectral curves for Willmore tori with
trivial normal bundle \cite{conformal_tori, schmidt}.

Note that the harmonicity (\ref{eq: conformal Gauss map harmonic}) of
the conformal Gauss map of a Willmore surface can be read as a
holomorphicity condition: decomposing the trivial connection $d= d_+ +
d_-$ into $S$ commuting and anti--commuting parts, the type
decompositions are
\[
d_+ =\partial + \delbar, \quad d_-=  A  + Q
\]
with complex holomorphic and anti--holomorphic structures $\partial$
and $\delbar$ and Hopf fields $A$ and $Q$.   Then 
\[
K\End_-(\trivial 2) =\{\omega\in\Omega^1(\End(\trivial 2)) \mid *\omega = S\omega = -\omega S\}
\]
can be equipped \cite{coimbra} with a complex holomorphic structure
$\delbar$ by putting
\[
(\delbar_X B)(Y)\phi = \delbar_X(B(Y)\phi) - B(\delbar_XY) \phi - B(Y) \partial_X\phi
\]
for $B\in \Gamma(K\End_-(\trivial 2))$ where $\phi\in\Gamma(\trivial
2)$, $X,Y\in \Gamma(TM)$, and $\delbar_X Y = \frac12([X,Y] + J[JX,
Y])$ is the complex holomorphic structure on $TM$.

Since the Hopf field $A$ of the conformal Gauss map of a Willmore
surface is a section of $K\End_-(\trivial 2)$, the harmonicity $d*A=0$
can be read as the condition that $A$ is a holomorphic section of
$\Gamma(K\End_-(\trivial 2))$ since
\[
d*A(X,J_MX) = -2(\delbar_X A)(X)\,.
\]
In particular, if $A\not=0$ the zeros of $A$ are isolated and the
kernel of $A$ defines a line subbundle $L_1 =\ker A$ of the trivial
$\H^2$ bundle. Similarly, the Hopf--field $Q$ is an anti--holomorphic
section of the bundle $\bar K\End_-(\trivial
2)=\{\omega\in\Omega^1(\End(\trivial 2)) \mid *\omega = -S\omega =
\omega S\}$, and the image of $Q$ defines a line subbundle $L_{-1}=
\im Q$ of $\trivial 2$.  The corresponding two maps $f_1, f_{-1}: M
\to S^4$ are either constant or branched conformal immersions
\cite{coimbra}.

\begin{definition}
  If $A\not=0$ then $\ker A = L_1$ is called the \emph{forward
    B\"acklund transform} of $f$ whereas for $Q\not=0$ we call $\im
  Q=L_{-1}$ the \emph{backward B\"acklund transform} of $f$.
\end{definition}

Since $A$ and $Q$ are essentially (\ref{eq:Hopf fields}) the $(1,0)$
and $(0,1)$ part of $dS$, the B\"acklund transformation is an analogue of
the $\delbar$ and $\partial$ construction for harmonic maps into
$\CP^n$.  To obtain a sequence of Willmore surfaces, we need to ensure
that the conformal Gauss map of a B\"acklund transforms extends
smoothly into the branch points of the B\"acklund transform. Indeed:

\begin{theorem}[\cite{sequences}]
  The conformal Gauss map of a B\"acklund transform $f_i$, $i=1,-1$,
  of a Willmore surface is a harmonic complex structure on $\trivial
  2$.  In particular, $f_i$ is a (branched) Willmore surface.
\end{theorem}

Thus, applying this procedure inductively, we obtain a sequence $f_i$,
$i\in \Z$, of Willmore surfaces which only breaks down if $A_i =0$,
$Q_i=0$ or $f_i$ is constant.  If the sequence breaks down at some
point $i$, then the sequence is in fact finite:

\begin{lemma}[\cite{sequences}]
  The possible sequences for a Willmore surface $f: M \to S^4$ are
  of the following form:
\begin{enumerate}
\item $
\circ \line \stackrel f
\bullet
\line \circ$,
\item $ \circ \line
\stackrel f
  \bullet \line ), ( \line  \stackrel f \bullet \line \circ$,
\item $
( \line \stackrel f \bullet \line )$
\item $
 ( \line \stackrel f \bullet \line \bullet
 \line )$, $( \line  \bullet \line \stackrel f \bullet
 \line )$, or
\item $
 \ldots  \line \bullet
\line
\stackrel f \bullet \line \bullet \line \bullet  \line \ldots  $
\end{enumerate}
where  $\bullet$ indicates a (non--constant) Willmore surface,
$\circ$ a point, and '')'' and ''(''  indicate that $A$ and $Q$
respectively are zero.
\end{lemma}

The finite sequences can be classified: we first note that $SL_i=L_i$
by (\ref{eq:A,Q}) for the B\"acklund transforms $f_i$, $i=1,-1$.  In
particular, if one of the B\"acklund transforms is a constant point
$f_i=\infty$, then the image of $f$ under the stereographic projection
across $\infty$ of $S^4=\R^4\cup \{\infty\}$ to $\R^4$ has mean
curvature sphere congruence $S$ with $\infty\in S(p)$ for all $p$. In
other words, the mean curvature sphere congruence of the surface in
$\R^4$ degenerates to a plane, and $f$ gives a minimal surface in
$\R^4$ under the stereographic projection. The minimal surface has
planar ends since $f_i=\infty\in L_q$ for some $q\in M$.  On the other
hand, in the case when $A=0$ then $f$ is the twistor projection of
holomorphic curves in $\CP^3$ \cite{coimbra}. Finally, if $Q=0$ then
$f$ is the dual surface of such a twistor projection.

\begin{lemma}[\cite{sequences}]
If the Willmore sequence of a Willmore surface $f: M \to S^4$ is finite then one of the following holds:
\begin{enumerate}
\item $f$ is after stereographic projection a minimal surface in $\R^4$ with planar ends, or
\item $f$ comes from the twistor projection of a complex holomorphic
  curve in $\CP^3$.
\end{enumerate}
\end{lemma}

  Since minimal surface are given by complex holomorphic functions via
  the Weierstrass representation, the previous lemma can be read as
  the statement that a Willmore surface with finite Willmore sequence
  is given by complex holomorphic data.

\begin{lemma}[\cite{willmore_tori}]
\label{eq: dual willmore}
  Let $f: M \to S^4$ be a Willmore surface which allows a dual
  Willmore surface, that is $AQ=0$. Then the Willmore sequence of $f$
  is finite if $f$ has normal bundle degree $|\deg \perp_f|> 4(g-1)$.
\end{lemma}

In particular, since Willmore spheres have dual surfaces, we recover
the results of Bryant \cite{bryant}, Montiel \cite{montiel} and Ejiri
\cite{ejiri} for Willmore spheres. More generally, we have:

\begin{theorem}[\cite{willmore_tori}, \cite{sequences}]
  If $f: S^2 \to S^4$ is a Willmore sphere, or if $f: T^2 \to S^4$ is
  a Willmore torus with non--trivial normal bundle, then $f$ is given
  by complex holomorphic data. More precisely, $f$ is after
  stereographic projection a minimal surface in $\R^4$ with planar
  ends or comes from the twistor projection of a complex holomorphic
  curve in $\CP^3$.
\end{theorem}
\begin{proof}
  As in the case of harmonic maps into $\CP^n$ the proof relies on an
  estimate on the energy of the harmonic map: If $f:M \to S^4$ is a
  Willmore surface with at least $n$ B\"acklund transforms, then the
  Pl\"ucker relation for quaternionic holomorphic curves
  \cite{klassiker} and a telescoping argument as in \cite{wolfson}
  give an estimate on the Willmore functional of $f$
\[
\tfrac{1}{4\pi}W(f)\ge 
-4n(n+1)(g-1) - n \, \deg \perp_f
\]
where $g$ is the genus of $M$.  In particular, in the case when $g=0$
 the right hand side 
\[
4n(n+1) - n \, \deg \perp_f 
\]
tends to $\infty$ as $n$ goes to $\infty$, contradicting the
finiteness of the Willmore energy. If $g=1$ then the leading term on
the right hand side is $-n\, \deg \perp_f$ and since $f$ has
non--trivial normal bundle we may assume, by passing to the dual
surface if necessary, that $\deg \perp_f<0$, which again gives a
contradiction.  Thus, in both cases the Willmore sequence is finite.
\end{proof}

In particular, lemma \ref{eq: dual willmore} and the previous theorem
are evidence for the following
\begin{conj}
  Let $f: M \to S^4$ be a Willmore surface of a Riemann surface $M$ of
  genus $g$ into the 4--sphere with $|\deg \perp_f|> 4(g-1)$. Then $f$
  is given by complex holomorphic data.
\end{conj}


\section{$\mu$--Darboux transforms of the conformal Gauss map}

In \cite{conformal_tori} the Darboux transformation on conformal
immersions from a Riemann surface into the 4--sphere is introduced.
In the case of constant mean curvature surfaces parallel sections of
the associated family of flat connections give Darboux transforms, the
so--called $\mu$--Darboux transforms.  These are  classical Darboux
transforms \cite{darboux} only if
$\mu\in\R_*\cup S^1$, but all $\mu$--Darboux transforms have constant
mean curvature \cite{cmc}.  In particular, the induced transformation
on the Gauss map again preserves harmonicity. We extend this
transformation on harmonic maps to the case when $S$ is the conformal
Gauss map of a Willmore surface and show that this is a transformation
on Willmore surfaces.

As usual for harmonic maps, we can introduce a spectral parameter and
characterise harmonic complex structures on $\trivial 2$ by the
flatness of a $\C_*$--family of flat connections on the trivial $\C^4$
bundle over $M$. Here, we equip $\H^2$ with the complex structure $I$
which is given by right multiplication by $i$, and thus identify
$(\H^2,I) =\C^4$ by $\H = \C + j\C$, $\C =\Span_\R\{1,i\}$.

\begin{theorem}
\label{thm:family of flat connections}
A complex structure $S\in\Gamma(\End(\trivial 2))$ is harmonic if and
only if the family of complex connections
\begin{eqnarray*}
d^\lambda &=& d +  (\lambda-1) A\oz + (\lambda\invers-1)A\zo
\end{eqnarray*}
on $M\times \C^4 = (\trivial 2, I)$ is flat. Here
\[
A\oz = \frac 12(A - I*A), \quad A\zo = \frac 12(A+I*A)
\]
denote the $(1,0)$ and $(0,1)$ parts of $A$ with respect to the
complex structure $I$.
\end{theorem}

\begin{proof}
  The proof is a standard calculation using $[I,S]=0$ and $Q\wedge A
  =0$ to obtain the curvature of $d^\lambda$ as
\begin{equation}
\label{eq:curvature family}
R^\lambda =
 (d*A)S\left((\lambda-1)\pi_E + (\lambda\invers-1)\pi_{E^\perp}\right)
\end{equation}
where  $E$ and $E^\perp$ denote the $\pm i$ eigenspaces of $S$ respectively, and
\[
\pi_E = \frac 12(1 - IS), \quad \pi_{E^\perp} = \frac 12(1+IS)
\]
the projections along the orthogonal splitting $\trivialC 4 = E \oplus
E^\perp$. Since (\ref{eq:curvature family}) holds for all
$\lambda\in\C_*$ we see that $d*A=0$ if and only if $R^\lambda=0$ for
all $\lambda\in\C_*$.
\end{proof}

At this point we are only interested in local theory, and thus will
assume from now on that $M$ is simply connected. Moreover, since a
Willmore surface whose conformal Gauss map has Hopf field $A=0$ is a
twistor projection of a holomorphic curve in $\CP^3$, we are primarily
interested in harmonic complex structures with Hopf field $A\not=0$.
However, the above family of flat connections is gauge equivalent to a
family of connections which are defined in terms of the Hopf field
$Q$, so similar arguments as in the following could be used to deal
with the case $A=0$, $Q\not=0$. If both Hopf fields vanish then $S$ is
constant.

\begin{theorem}
\label{thm:Darboux on conformal Gauss map}
Let $S$ be a harmonic complex structure on $\trivial 2=M\times \H^2$
with $A\not=0$ and $d^\lambda $ the associated $\C_*$ family of
complex connections on $(\trivial 2, I)$. For fixed $\mu\in\C_*$ let
$\psi_1,\psi_2\in\Gamma(\trivial 2)$ be parallel sections of $d^\mu$
such that $W_\mu=\Span_\C\{\psi_1, \psi_2\}$ is a complex rank 2
bundle over $M$ with $W_\mu\cap W_\mu j=\{0\}$. Then
\begin{equation}
\label{eq:T} 
T = S(a-1) + b
\end{equation}
is invertible on $M$ for $\mu\in \C_*$, $\mu\not=1$, where $ a =
G\left(\frac{\mu+\mu\invers}2E_2\right) G\invers$, $ b = G\left(
  I\frac{\mu\invers-\mu}2\right) G\invers$ with identity matrix
$E_2\in\Gl(2,\H)$ and $G=
(\psi_1,\psi_2)\in\Gamma(\Gl(2,\H))$. Moreover,
\[
\hat S = T\invers S T
\]
is harmonic with Hopf fields
\begin{equation}
\label{eq:hopf fields darboux}
*\hat A = \frac 12 T (1-a)\invers *A T, \quad *\hat Q = 2 T\invers
*Q(a-1) T\invers\,.
\end{equation}
We call $\hat S$ the \emph{$\mu$--Darboux--transform} of the harmonic
complex structure $S$.  Note that $\hat S$ is independent of the
choice of basis of $W_\mu$.
\end{theorem}
\begin{proof}
  By assumption $\trivialC 4=W_\mu \oplus W_\mu j$ so that every
  $\phi\in\trivialC 4$ has a unique decomposition $\phi=\phi_1 +
  \phi_2 j$ with $\phi_l\in W_\mu$. Decomposing $\phi_l=
  \phi_l^++\phi_l^{-}$ further according to the splitting $\trivialC 4
  = E\oplus Ej$ with $\phi_l^+\in E, \phi_l^{-} \in E^\perp$ for
  $l=1,2$, we have a unique decomposition
  \[
\phi=\phi_1^++\phi_1^{-} + (\phi_2^+ + \phi_2^{-}) j\,.
\]
Applying  (\ref{eq:T}) we get
\[
T \phi = \phi_1^+i(\mu\invers-1) + \phi_1^{-}i(1-\mu) +
\phi_2^+ji(1-\bar \mu) + \phi_2^{-}ji(\bar\mu\invers-1)\,,
\]
and $T\phi=0$ implies $ \phi_1^+i(\mu\invers-1) = \phi_1^{-}i(1-\mu) =
\phi_2^+ji(1-\bar \mu) = \phi_2^{-}ji(\bar\mu\invers-1)=0$ because
$\trivialC 4 = W_\mu\oplus W_\mu = E\oplus Ej$. Since $\mu\not=1$ this
shows that $\phi=0$, and thus $T$ is invertible.

The remainder of the proof is the exact analogue for the corresponding
statement for harmonic complex structures $J$ on $\trivial{}$, see
\cite{cmc}: Since $ \trivial 2 = W_\mu \oplus W_\mu j$ and $W_\mu$ is
$d_\mu$--parallel, the quaternionic extension of $d_\mu|_{W_\mu}$ is
the quaternionic connection
  \[
\hat d^\mu = d+ *A T
\]
on $\trivial 2$.  Moreover, $a,b$ are constant on the basis
$\{\psi_1,\psi_2\}$ of $W_\mu$ so that $\hat d_\mu a = \hat d_\mu
b=0$, or expressed differently,
\begin{equation}
\label{eq:dI}
d(a-1)=- [*AT, a-1]\,, \quad db = -[*AT, b] \,.
\end{equation}
Differentiating (\ref{eq:T}) we obtain the Riccati type equation
\begin{equation}
 dT  =
 *Q( a-1) + 2 T *AT\label{eq:Ricatti}\,,
\end{equation}
where we  used (\ref{eq:Hopf fields}) and $a^2+b^2=E_2$.
Therefore, the derivative $d\hat S = [\hat S, T\invers dT] + T\invers dS T$
of $\hat S$ computes to
\begin{equation*}
d\hat S = -2 T\invers* Q\big((\hat a-1) \hat S +S(\hat a-1)\big) + 2(S +\hat S)*AT + T\invers d S T
\end{equation*} 
and the Hopf fields to
\[
-2*\hat A = (S+\hat S - 2T\invers)*AT\,, \quad *\hat Q  = T\invers *Q( - (a-1)\hat S - S(a-1) + T)\,.
\]
Now $-2(a-1)+ ST(a-1) - Tb=0$ by (\ref{eq:T}), and thus
\begin{equation}
\label{eq:hat S}
\hat S = 2T\invers + \frac b{a-1}\,
\end{equation}
give the equations (\ref{eq:hopf fields darboux}). Finally, (\ref{eq:Ricatti}) shows $dT\invers= - 2T\invers *Q(a-1) T\invers - *A$, and 
\[
*\hat Q =- dT\invers - *A
\]
is closed since $S$ is harmonic. In other words, $\hat S$ is harmonic
by (\ref{eq: conformal Gauss map harmonic}).
\end{proof}

\begin{rem}
\label{rem: trivial DTs}
For $\mu\in S^1$ the $\mu$--Darboux transform is trivial: in this case
$a$ and $b$ are real multiples of the identity, and therefore $\hat
S=S$ since $[T,S]=0$.
\end{rem}

The $\mu$--Darboux transformation preserves the Willmore property:

\begin{theorem} 
\label{thm: Darboux transform}
Let $f: M \to S^4$ be a Willmore surface which is not a twistor
projection of a holomorphic curve in $\CP^3$. Let $S$ be the conformal
Gauss map of $f$, then the $\mu$--Darboux transform $\hat S$ of $S$ is the
conformal Gauss map of a Willmore surface $\hat f$.
\end{theorem} 

\begin{proof}
  Let $\hat S = TST\invers$ be a $\mu$--Darboux transform of $S$ where
  $T$ is defined (\ref{eq:T}) for two parallel sections of $d^\mu$
  satisfying the assumptions of Theorem \ref{thm:Darboux on conformal
    Gauss map}. We show that $\hat S$ is the conformal Gauss map of
\[
\hat L =  T(a-1)\invers L\,.
\]
With $a^2+b^2=E_2$ and (\ref{eq:hat S}) we obtain $\hat S
T(a-1)\invers = T(a-1)\invers S$ and thus $\hat L$ is $\hat S$ stable
since $S L =L$.  Furthermore, by (\ref{eq:hopf fields darboux})
\begin{equation}
\label{eq:kerim}
\im \hat A \subset \hat L \subset \ker\hat Q 
\end{equation}
since $\im A \subset L\subset \ker Q$. In particular, $d\hat S =
2(*\hat Q -*\hat A)$ stabilises $\hat L$ which gives $\hat S \hat
\delta = \hat\delta \hat S$.  Finally, for $\hat\psi\in\Gamma(\hat L)$
we see
\[
 \hat\delta \wedge *\hat A \hat\psi = \pi_{\hat L} d \wedge (*\hat A\hat\psi)=0
\]
since $\hat S$ is harmonic and $\im \hat A \subset \hat L$. Since
$\hat A\not=0$ by (\ref{eq:hopf fields darboux}) this shows that $\hat
\delta$ has type $*\hat \delta = \hat \delta\hat S$.
\end{proof}

\begin{rem}
\label{rem:coimbra versus mu}
The Darboux transformation as defined in \cite{coimbra} is the special
case $\mu\in S^1\cup\R_*$, $\mu\not=1$ of our $\mu$--Darboux
transformation. In this case, the map $a$ is a real multiple of the
identity, and the Riccati type equation (\ref{eq:Ricatti}) becomes
the Riccati equation as in \cite{coimbra} and (\ref{eq:hat S}) becomes
the corresponding initial condition.
\end{rem}

\begin{rem} A similar theorem holds for constrained Willmore surfaces
  \cite{les}: an immersion $f: M \to S^4$ is constrained Willmore
  \cite{bpp_constrained} if $d(*A+\eta)=0$ where $S$ is the conformal
  Gauss map of the constrained Willmore immersion $f: M \to S^4$ and
  $\eta\in\Omega^1(\Hom(\trivial 2/L, L))$ is the potential of $f$ with
  $*\eta=S\eta= \eta S$. The complex structure $S$ then gives rise to
  a family of flat connections $d^\lambda= d + (*A+\eta)(S(a-1) + b)
  -2*\eta(a-1)$ with $a=E_2\frac{\lambda+\lambda\invers}2,
  b=I\frac{\lambda\invers-\lambda}2$, $\lambda\in\C_*$. Given two
  parallel section of $d_\lambda$ we define $T$ by (\ref{eq:T}) and
  the complex structure by $\hat S = T\invers ST$. Then $\hat S$ is
  the conformal Gauss map of a constrained Willmore surface with
  potential $\hat \eta =T(a-1)\invers \eta (a-1)T\invers$.
\end{rem}

In \cite{cmc_dressing} it is shown that the $\mu$--Darboux transform
of the Gauss map of a constant mean curvature surface is given by a
simple factor dressing. In \cite{quintino} a dressing transformation
on Willmore surfaces, and more generally on constrained Willmore
surfaces, is introduced. Moreover, Quintino shows that the Darboux
transforms as defined in \cite{coimbra}, that is by Remark
\ref{rem:coimbra versus mu} the $\mu$--Darboux transforms of the
conformal Gauss map with $\mu\in\R_*\cup S^1$, are given by a simple
factor dressing. We expect \cite{les} that the computations in
\cite{cmc_dressing} transfer to the Willmore case, and that a simple
factor dressing of a harmonic complex structure is indeed a
$\mu$--Darboux transform for $\mu\in\C_*$.

In \cite{conformal_tori} a Darboux transform on conformal tori is
defined which, even in the case of a Willmore surface, differs from
the Darboux transform on the conformal Gauss map in Theorem \ref{thm:
  Darboux transform}: Recall that for a conformal immersion $f: M \to
S^4$ a section $\psi^\sharp \in\Gamma(L^\sharp)$ defines a
\emph{Darboux transform} $L^\sharp = \psi^\sharp\H$ of $f$ if
$d\psi\in\Omega^1(L)$ where $L$ is the line bundle of $f$. In
particular, if $f$ is Willmore and $\psi_1, \psi_2$ are parallel
sections of $d^\mu$ for fixed $\mu\in\C_*$ as in Theorem
\ref{thm:Darboux on conformal Gauss map}, we have
\begin{equation}
\label{eq:dpsi l}
d\psi_l = - *AT\psi_l \in \Omega^1(L)\,, \quad l=1,2\,,
\end{equation}
since $\im A \subset L$. Thus $L^l = \psi_l\H$ are Darboux transforms
of $f$. For $\mu\in S^1$ we see with $[T,S]=0$ that $*d\psi_l =
*ATS\psi_l$, that is $*\delta_l\psi_l = -\delta_i S\psi_l$. In
particular, see Remark \ref{rem: trivial DTs}, the $\mu$--Darboux
transform $\hat S = S$ of $S$ is not the conformal Gauss map of $L^l$.

Since the $\mu$--Darboux transformation is defined by two parallel
sections and gives a change of sign on the complex structure in the
case when $\mu\in S^1$, we rather expect to see that a $\mu$--Darboux
transform is the conformal Gauss map of a two--fold Darboux transform of
the backward B\"acklund transform; we will return to this topic in a
future paper.


\bibliographystyle{alpha}
\bibliography{doc}

\end {document}